\documentclass[11pt,reqno]{amsart}
\usepackage{amsmath}
\usepackage{amssymb,amscd,verbatim,graphicx,color,amsthm}

\begin{document}

\newcommand\vh{\vspace{0.2in}}
\newfont{\sans}{cmss10}
\newfont{\sansm}{cmss8}
\newcommand\al{\alpha}
\newcommand\chal{\check\alpha}
\def\ds{\displaystyle}
\def\g{\mathfrak g}
\def\m{\mathfrak m}
\def\a{\mathfrak a}
\def\b{\mathfrak b}
\def\z{\mathfrak z}
\def\<{\langle}
\def\>{\rangle}
\def\wt{\widetilde}
\def\wh{\widehat}
\def\om{\omega}
\newcommand{\dia}[8]{{#1-#3-&#4&-#5-#6-#7-#8\\ &|& \\ &#2&}}
\newcommand{\ovl}{{\overline }}
\newcommand{\C}[1]{{\mathcal #1}}
\newcommand\CA{{\mathcal A}}
\newcommand\CH{{\mathcal H}}
\newcommand\CI{{\mathcal I}}
\newcommand\CK{{\mathcal K}}
\newcommand\CL{{\mathcal L}}
\newcommand\CN{{\mathcal N}}
\newcommand\CO{{\mathcal O}}
\newcommand\CG{{\mathcal G}}
\newcommand\CP{{\mathcal P}}
\newcommand\CR{{\mathcal R}}
\newcommand\CS{{\mathcal S}}
\newcommand\CU{{\mathcal U}}
\newcommand\CV{{\mathcal V}}
\newcommand\CW{{\mathcal W}}
\newcommand\CX{{\mathcal X}}
\newcommand\CY{{\mathcal Y}}
\newcommand\bB{{\mathbb B}}
\newcommand\bC{{\mathbb C}}
\newcommand\bF{{\mathbb F}}
\newcommand\bH{{\mathbb H}}
\newcommand\bK{{\mathbb K}}
\newcommand\bL{{\mathbb L}}
\newcommand\bN{{\mathbb N}}
\newcommand\bR{{\mathbb R}}
\newcommand\bW{{\mathbb W}}
\newcommand\bZ{{\mathbb Z}}
\newcommand\ve{{\ ^\vee ~}}
\newcommand\one{{{1\!\!1}}}
\newcommand\zero{{{0\!\!0}}}
\newcommand\ovr{\overline}
\newcommand\la{{\lambda}}
\newcommand\ep{{\epsilon}}
\newcommand\sig{{\sigma}}
\newcommand\ome{{\omega}}
\newcommand{\fk}[1]{{\mathfrak #1}}
\newcommand\fc{{\mathfrak c}}
\newcommand\fg{{\mathfrak g}}
\newcommand\fh{{\mathfrak h}}
\newcommand\frk{{\mathfrak k}}
\newcommand\fl{{\mathfrak l}}
\newcommand\fm{{\mathfrak m}}
\newcommand\fn{{\mathfrak n}}
\newcommand\ovlfn{\overline{\mathfrak n}}
\newcommand\fo{{\mathfrak o}}
\newcommand\fp{{\mathfrak p}}
\newcommand\fq{{\mathfrak q}}
\newcommand\fs{{\mathfrak s}}
\newcommand\ft{{\mathfrak t}}
\newcommand\fu{{\mathfrak u}}
\newcommand\fz{{\mathfrak z}}
\newcommand\fa{{\mathfrak a}}
\newcommand\fb{{\mathfrak b}}
\newcommand\fC{{\mathfrak c}}
\newcommand\fO{{\mathfrak O}}
\newcommand\gc{{\fg_c}}
\newcommand\kc{{\fk_c}}
\newcommand\lc{{\fl_c}}
\newcommand\qc{{\fq_c}}
\newcommand\sce{{\fs_c}}
\newcommand\tc{{\ft_c}}
\newcommand\uc{{\fu_c}}
\newcommand\ssk{{\mathbb K}}
\newcommand\ssl{{\mathbb L}}
\newcommand\tie{{\tilde e}}
\newcommand\wti{\widetilde }
\newcommand\what{\widehat}
\newcommand\wmuO{{\widehat{\mu_\CO}}}
\newcommand\COm{{\CO_\fm}}
\newcommand\tCOm{{\widetilde{\COm}}}
\newcommand\CSA{Cartan subalgebra ~}
\newcommand\CSG{Cartan subgroup ~}
\newcommand\phih{{\phi^{\fh}_f}}
\newcommand\ie{{\it i.e. ~}}
\newcommand\eg{{\it e.g. ~}}
\newcommand\cf{{\it cf. ~}}
\newtheorem{theorem}{Theorem}[section]
\newtheorem{corollary}[theorem]{Corollary}
\newtheorem{lemma}[theorem]{Lemma}
\newtheorem{proposition}[theorem]{Proposition}
\newtheorem{conjecture}[theorem]{Conjecture}
\newtheorem{definition}[theorem]{Definition}
\newtheorem{remark}[theorem]{Remark}
\newtheorem{example}[theorem]{Example}

\newcommand\Ad{{\operatorname{Ad}}}
\newcommand\Spec{{\operatorname{Spec}}}
\newcommand\ad{{\operatorname{ad}}}
\newcommand\tr{{\operatorname{tr}}}
\newcommand\Hom{{\operatorname{Hom}}}
\newcommand\End{{\operatorname{End}}}
\newcommand\Tr{{\operatorname{Tr}}}
\newcommand\vsp{{\vspace{0.5in}}}

\numberwithin{equation}{section}

\title
[Ramanujan bigraphs]{Ramanujan bigraphs  associated with $SU(3)$ over a $p$-adic field}
\author{Cristina Ballantine and Dan Ciubotaru}
\address{Department of Mathematics and Computer Science, College of
  the Holy Cross, Worcester, MA 01610}
\email{cballant@holycross.edu}

\address{Department of Mathematics, University of Utah, Salt Lake
 City, UT 84112} 

\email{ciubo@math.utah.edu}

\date{\today}

\begin{abstract} We use the representation theory of
   the quasisplit form $G$ of $SU(3)$ over a $p$-adic field to investigate whether certain quotients of the Bruhat--Tits tree associated to 
this form are Ramanujan bigraphs. We show that a quotient of the tree
associated with  $G$ (which is a biregular bigraph) is Ramanujan if and only if $G$ satisfies a Ramanujan type conjecture. This result is analogous to the seminal case of $PGL_2(\mathbb{Q}_p)$ considered in \cite{lps}. As a consequence, the classification of the automorphic spectrum of the unitary group in three variables in \cite{R} implies the existence of certain infinite families of Ramanujan bigraphs. \end{abstract}

\maketitle

\section{Introduction}

This paper is concerned with the question of constructing infinite
families of Ramanujan bigraphs, \emph{i.e.}, biregular bipartite
graphs with certain strong connectivity conditions. Following the
classical case of regular graphs, the condition is that the second
largest eigenvalue of the adjacency matrix is asymptotically as small
as possible. The notion of Ramanujan bigraph is made precise in
Definition \ref{d:ram}, and it is motivated by the result of \cite{fl} (see Proposition \ref{bound}), which is the analogue for bigraphs of the well-known bound of Alon-Boppana for regular graphs. There is also an important notion of a weak Ramanujan bigraph which has to do with the multiplicity of the eigenvalue $0$ in the adjacency matrix of the bigraph. Every Ramanujan bigraph is immediately weakly Ramanujan.

Following the approach of \cite{lps}, where infinite families of regular Ramanujan graphs were contructed using the representation theory of $PGL(2,\mathbb Q_p)$ and deep results from the theory of automorphic forms, we construct families of Ramanujan bigraphs using the quasisplit form $G$ of  the group $SU(3)$ defined over a $p$-adic field $K$. For the group $G$, the Bruhat-Tits tree $\tilde X$ is an infinite biregular bipartite bigraph with valencies $(q^3+1,q+1)$, where $q$ is the order of the residue field of $K$. Let $I$ denote the stabilizer in $G$ of an edge of $\tilde X$. This is an Iwahori subgroup of $G$. Assume that we have a discrete co-compact subgroup $\Gamma$ of $G$ which acts without fixed points. Our first result (Theorem \ref{t:crit}) shows that the quotient tree $X=\tilde X/\Gamma$ is a Ramanujan bigraph if and only if a certain Ramanujan type conjecture is satified. More precisely, we have:

\begin{theorem} The quotient tree $X=\tilde X/\Gamma$ is a Ramanujan bigraph if and only if every nontrivial irreducible unitary representation of $G$ with Iwahori fixed vectors that appears in $L^2(G/\Gamma)$ is tempered.
\end{theorem}

To establish this equivalence, we make use of the Iwahori-Hecke algebra $\mathcal H(G,I)$, its representation theory (see \S 3) and its action on $X$ via adjancency operators (see \S 4, following \cite{H}).

In order to use this criterion to show the existence of families of Ramanujan bigraphs, we rely on the results of \cite{R}, where the classification of the automorphic spectrum of $U(3)$ is completed. Let $\mathbb G$ be an inner form of $SU(3)$ defined over a global field such that at a place $v$, the group $\mathbb G_v$ of points over the local field  is isomorphic to $G$. The group $\mathbb G$ is constructed from a central simple algebra $D$ of degree $3$ over a quadratic extension $E$ of the global field, by means of an involution $\alpha$ of the second kind.  We construct an infinite family of discrete co-compact subgroups $\Gamma_{i,v,n}$ of $G$ ($i$ ranges over a finite set, $n$ over an infinite set, see \S\ref{sec:3.3}) and consider the family of quotient trees $X_{i,v,n}=\tilde X/\Gamma_{i,v,n}$. Translating the results of \cite{R} in this setting, one finds the following result.

\begin{theorem} Assume that at the infinite places, $\mathbb G_\infty$ is compact. Then $X_{i,v,n}$ is a Ramanujan bigraph if and only if it is a weakly Ramanujan bigraph. If one assumes that the central simple algebra $D$ (which gives rise to $\mathbb G$) is not split over $E$, then the graphs $X_{i,v,n}$ are Ramanujan bigraphs.
\end{theorem}

The organization of the paper is as follows. In \S2, we recall the
basic definitions about Ramanujan graphs and bigraphs and the
associated zeta functions. In \S3, we introduce the Iwahori Hecke
algebra of the quasisplit $SU(3)$ and present the classification of
its modules, including the unitary dual. In \S4, we explain how the spectral theory of the Iwahori-Hecke algebra is related to the eigenvalues of the zeta function of the Bruhat-Tits bigraph, and obtain the main results listed above.

\smallskip

\noindent{\bf Acknowledgments.} {The authors thank Jon Rogawski for suggesting the problem and Gordan Savin for helpful conversations about
  the material in \S\ref{sec:3.3}.} D.C. was supported in part by
NSA-AMS 081022.

\section{Preliminaries}

\subsection{}As motivation, we first recall briefly the notion of regular Ramanujan graphs. For any $k$-regular graph $X$, we denote by $V(X)$ its vertex set and by $E(X)$ its edge set. We denote by $\Ad(X)$ the adjacency matrix of $X$ and by $\Spec(X)$ the spectrum of $X$, \emph{i.e.}, the collection of eigenvalues of $\Ad(X)$. Notice that we have $k\in\Spec(X)$. Denote by \begin{equation}\lambda(X)= \max \{|\lambda| \, : \, \lambda \in \mbox{Spec}(X), |\lambda |\neq k\},\end{equation} the second largest  eigenvalue of $\Ad(X)$ in absolute value.  If $A$ is a collection of vertices
of
$G$, the {boundary} of $A$, denoted $\partial
A$, is given by $\partial A = \{ x\in V(X)  \setminus  A :~ \{x,y \} \in E,\ \mbox{for some} \ y \in A \}$.

\begin{definition} A finite regular graph $X$ on $n$ vertices and of degree $k$ is called an $(n,k,c)$-expander if for every  subset $A$ of $V(X)$ with 
$\displaystyle  |A| \leq \frac{n}{2}$ we have $|\partial  A| \geq  c |A|.$ The
constant 
 $c$ is called the {expansion coefficient}.
\end{definition}

The expansion coefficient of a regular graph $X$ depends on $\lambda(X)$. More precisely, it is known (\cite{lps}) that a finite $k$-regular graph $X$ on $n$
vertices is  an $(n,k,c)$-expander with $2c=1-\lambda (X)/ k$.
Thus, good expanders have small $\lambda(X)$. However, asymptotically $\lambda(X)$ cannot
be made arbitrarily small, as known from the classical result of Alon-Boppana, which shows that if $X_{n,k}$ is a $k$-regular graph on $n$ vertices, then
\begin{equation}
{\displaystyle \lim_{n
\rightarrow
\infty }} 
\lambda(X_{n,k})\geq 2 \sqrt{k-1}.
\end{equation}

This bound leads to the following natural definition.

\begin{definition}[\cite{lps}] A $(q+1)$-regular graph $X$  is called a
\emph {Ramanujan graph} if  $\lambda(X)\leq 2\sqrt q.$
\end{definition}

Infinite families of Ramanujan graphs of constant degree have been
constructed using the Bruhat-Tits tree for $PGL(2,\mathbb Q_p)$ and deep results from number theory \cite{lps,m}. We will be interested in the analogous notion and constructions for a Ramanujan bigraph.

\subsection{}
Recall that a {$(k,l)$-biregular bigraph} $X$ is
a bipartite graph in which all vertices of one color have degree $k$ and
all vertices of the other color have degree $l$.
If $X$ is a $(k,l)$-biregular bigraph, the trivial eigenvalues of its
adjacency matrix  are
$\pm \sqrt{kl}$. As before, let
\begin{equation}\lambda (X)=\max \{|\lambda| \, : \, \lambda \in \mbox{Spec}(X), |\lambda |\neq \sqrt{kl}\}
\end{equation} denote the absolute value of
the second largest eigenvalue, in absolute value, of $X$. Then,
analogous to the Alon-Boppana bound we have the following result due to
Feng and Li ~\cite{fl}.

\begin{proposition}[\cite{fl}]\label{bound} If $X_{k,l,n}$ is a $(k,l)$-biregular bigraph on $n$
vertices, then  
$${\displaystyle \liminf_{n \rightarrow \infty}} \ \lambda(X_{k,l,n})
\ge
\sqrt{k-1} +
\sqrt{l-1}.$$
\end{proposition}

Assume now that we have a finite, connected, biregular, bipartite graph $X$ with valencies $q_1+1$ and $q_2+1$, and assume that $q_1\ge q_2$. Let $n_i$ denote the number of vertices of valency $q_i+1$, $i=1,2$, and set  $V(X)=V_1(X) \sqcup V_2(X)$, where $V_i(X)$ consists of the
vertices of $X$ with valency $q_i+1$. Then necessarily we have $n_2\ge n_1$.  The adjacency matrix of the graph $X$ has the eigenvalues 
\begin{equation}\label{spec}
\begin{aligned}
\Spec(X) = \{
\pm \lambda_1, \pm \lambda_2, \ldots \pm \lambda_{n_1}, \underbrace{0, \ldots , 0}_{n_2-n_1}\},\\
\lambda_1=\sqrt{(1+q_1)(1+q_2)}>\lambda_2 \geq \cdots \geq \lambda_{n_1}\geq
0.
\end{aligned}
\end{equation}

Using the bound in Proposition \ref{bound}, Hashimoto \cite{H} (also Sol\'e \cite{S}) defines Ramanujan bigraphs as follows.

\begin{definition}\label{d:ram}A finite, connected, biregular, bipartite
graph with valencies $q_1+1$ and $q_2+1$ is called \emph{Ramanujan bigraph} if
\begin{equation}\label{ramanujan}
|\lambda^2-q_1-q_2| \leq 2 \sqrt{q_1q_2},
\end{equation} for every $\lambda \in \{\pm \lambda_2,\pm\lambda_3,\dots,\pm \lambda_{n_1}\}$, where the notation is as in (\ref{spec}).
In particular, this means $\lambda_{n_1}>0$.

 We call $X$ a \emph{weak Ramanujan bigraph} if $\lambda_{n_1}>0$, in other words if $\Ad(X)$ has the eigenvalue $0$ with multiplicity exactly $n_2-n_1.$ 
\end{definition}

Next, we give an equivalent characterization for Ramanujan bigraphs.

\subsection{} A cycle in a graph $X$ is an equivalence class of a closed  path, where the equivalence is given by shifting the origin. Multiplication of paths is defined by concatenation. A cycle is primitive if it is not a power of a shorter cycle.  If $X$ is a graph with fundamental group $\Gamma$, and $\rho: \Gamma \rightarrow U(n)$ is an $n$-dimensional unitary
representation of $\Gamma$ we define, following \cite[\S 7]{H-H}, the Zeta function of
the graph $X$ attached to $\rho$.

\begin{definition}\label{zeta} The Zeta function of
the graph $X$ attached to a representation  $\rho$ of its fundamental group is given by \begin{equation}Z_{X}(u;\rho)=\prod_{\stackrel{{\displaystyle
P} \ \mbox{\small primitive}}{\mbox{\small cycle in } X}} \det(I_n-\rho(\langle P
\rangle)u^{|P|/2})^{-1},
\end{equation}
 where $|P|$ is the length of the geodesic cycle $P$ in
$x$ and $\langle P \rangle$ is the corresponding conjugacy class in $\Gamma$.  We
denote the Zeta function of $X$ attached to the trivial representation of $\Gamma$ by $Z_X(u)$. \end{definition}

\smallskip

There is an explicit description in \cite{H} of $Z_X(u)$ for a $(q_1+1, q_2+1)$--biregular graph $X$.  

\begin{theorem}[Main Theorem (III),\cite{H}] We have
\begin{equation}\label{prod} 
Z_X(u)^{-1}=(1-u)^{r-1}(1+q_2u)^{n_2-n_1} \prod_{j=1}^{n_1}\left(1-(\lambda_j^2
-q_1-q_2)u+q_1q_2u^2\right),\end{equation}
 where $r=$ rank of $\Gamma= |E(X)|-|V(X)|+1$.
\end{theorem}

The {trivial zeros} of $Z_X(u)^{-1}$ are  $\displaystyle1, (q_1q_2)^{-1}, -{q_2}^{-1}$. 

\begin{definition} The Zeta function of a finite, connected, biregular,
bipartite graph $X$ with valencies $q_1+1$ and $q_2+1$ is said to satisfy the
\emph{Riemann Hypothesis} if the  non-trivial zeros   of $Z_X(u)^{-1}$ satisfy the following property: 
 \begin{equation}\label{riemann}
\text{if $\Re(s)\in (0,1)$ and $Z_X((q_1q_2)^{-s})^{-1}=0$,  then
$\Re(s)=1/2$.}\end{equation}\end{definition}

The following result shows that the notion of Ramanujan bigraph and Riemann Hypothesis are equivalent. The proof is elementary.

\begin{lemma} A finite, connected, biregular, bipartite graph is
Ramanujan  if and only if its Zeta function satisfies the Riemann Hypothesis.
\end{lemma}

\begin{proof} Let  $X$ be a $(q_1+1,q_2+1)$-biregular bipartite graph. If $X$ is Ramanujan, then
$(\lambda^2 -q_1-q_2)^2-4q_1q_2 \leq 0,$ for $\lambda \in \mbox{Spec}(X)$ with
$\lambda^2 \neq (1+q_1)(1+q_2)$. Condition (\ref{ramanujan}) implies  that each
nonlinear, nontrivial term of $Z^{-1}_X(u)$ has complex conjugate roots of modulus
$(q_1q_2)^{-1/2}$. Thus, since
$|(q_1q_2)^{-s}|=(q_1q_2)^{-\Re(s)}$, $X$  Ramanujan implies that the complex
solutions of $Z_X(u)$ are of the form $(q_1q_2)^{-s}$ with
$\Re(s)=\displaystyle{\frac{1}{2}}$. The real solutions of $Z_X(u)$ are $1$,
$-q_2^{-1}$, $(q_1q_2)^{-1}$  (the trivial solutions), and $\pm
(q_1q_2)^{-1/2}$ (if any eigenvalue satisfies
the equality in (\ref{ramanujan})). Thus, $Z_X$ satisfies (\ref{riemann}).

Conversely, assume that $Z_X$ satisfies (\ref{riemann}), but
suppose that $X$ is
not Ramanujan. Then, there exists $\lambda \in \mbox{Spec}(X)$ such that $\lambda^2 <
(q_1+1)(q_2+1)$ and  $|\lambda^2-q_1-q_2| > 2 \sqrt{q_1q_2}.$ This implies
that the term $1-(\lambda^2-q_1-q_2)u+q_1q_2u^2$ in $Z_X(u)^{-1}$ has two real
(nonequal) roots $u_{1,2}$. There are two cases.  

The first case is $\lambda^2-q_1-q_2 > 2 \sqrt{q_1q_2}$.
Then $u_{1,2}$ are positive. We write the roots in the form $(q_1q_2)^{-s}$, where
$s$ is real. Then we obtain
$\lambda^2-q_1-q_2=(q_1q_2)^{1-s}+(q_1q_2)^{s}$. The right hand side, as a function
of $s$, has a minimum of $2\sqrt{q_1q_2}$ at $s=1/2$. Thus,
if $s\geq 1$, respectively $s\leq 0$, the right hand side is increasing, respectively decreasing  and we have $\lambda^2-q_1-q_2=(q_1q_2)^{1-s}+(q_1q_2)^{s}
\geq 1+q_1q_2$. This contradicts our assumption that $\lambda^2 < (q_1+1)(q_2+1)$.
Thus, $s \in (0,1)$. Since the roots are positive, 
(\ref{riemann}) implies that $s=\Re(s)=1/2$ and we have $(\lambda^2-q_1-q_2)^2=4q_1q_2$. This however contradicts our assumption that
$|\lambda^2-q_1-q_2| > 2 \sqrt{q_1q_2}$.

The second case, $\lambda^2-q_1-q_2 < -2 \sqrt{q_1q_2}$, is analogous.
\end{proof}

\section{The Iwahori-Hecke algebra of $SU(3)$}

\subsection{}\label{sec:su3}
Let $K$ be a local field  with discrete valuation $\omega$. Let $q$ be
the cardinality of the residue field. Let $L/K$ be an unramified\footnote{If $L/K$ is ramified, the building of $SU(3)$ is a 
$(q+1)$-regular tree and this case falls into the framework of \cite{lps}.}
separable quadratic extension.
Let $\varpi$ be the uniformizer of $K$ and $\varpi_1$ the uniformizer of $L$. Let
$\Xi=\omega(K^{\times})$ ($\subset \mathbb{R}$) be the value group of $K$. Since
$L/K$ is unramified, $\Xi$ is also the value group of $L$. Consider
the hermitian form $\Phi= \left(
\begin{array}{ccc}0 & 0 & 1\\0 & -1 & 0 \\ 1 & 0 & 0 \end{array}\right)$.
Set
\begin{equation}
G=SU(3) =\{ g \in SL_3(L)\, | \, g \, \Phi \, {\bar{g}}^T=\Phi \}.
\end{equation}
Let $S$ be a maximal torus in $G$ defined by 
\begin{equation}
S(K)=\left\{\left.\left(
\begin{array}{ccc}d & 0 & 0\\0 & 1 & 0 \\ 0 & 0 & d^{-1} \end{array}\right)\, \right|\, d
\in K^{\times} \right\}.
\end{equation}
The root datum of $(G,S)$ is $(X, X^\vee,\Delta, \Delta^\vee)$. 
Here $X=X^*(S)=\Hom(S, K^\times)$, $X^{\vee}=X_*(S)=\Hom(K^\times, S)$ and  $\Delta=\{a_1, a_{-1},
2a_1, 2 a_{-1}\}$, where $na_{\pm1}:S \rightarrow K^\times\ $ is defined by
\begin{equation}
na_{\pm 1}\left(\left(\begin{array}{ccc}t & 0 & 0\\0 & 1 & 0 \\ 0 & 0 & t^{-1}
\end{array}\right)\right)=t^{\pm n},\ n=1,2.
\end{equation} 
The affine roots are 
\begin{equation}
\phi_{\mathsf{af}} = \{2a_i+ \gamma |\, i=\pm 1, \gamma \in
\Xi\}\cup \{a_i+ \frac{1}{2}\gamma |\, i=\pm 1, \gamma \in
2\Xi\}.
\end{equation}
The inequalities $0 <a_1<\omega(\varpi_1)$ define a chamber  and the
corresponding basis is $\{a_1, 2a_{-1}+\omega(\varpi_1)\}$. Let
$\{s_1,s_2\}$ denote the corresponding reflections. They generate the
affine Weyl group $W_{\mathsf{af}}$, the infinite dihedral group.

Let $v_1$ be the vertex in the Dynkin diagram corresponding to the simple root $a_1$, and $v_2$ be
the vertex in the Dynkin diagram corresponding to the simple root
$2a_{-1}+\omega(\varpi_1)$. Following \cite{C}, the parameters of the root system are the integers
 $d(v_1)=3$ and $d(v_2)=1$. Set $q_i=q^{d(v_i)}$, $i=1,2$.

\subsection{}\label{L2} Recall that a $G$-representation is called unitary if it is
defined on a Hilbert space such that the inner product is
$G$-invariant. Assume that $G$ acts on a space $Y$ which has a
$G$-invariant measure. Then we can consider the Hilbert space $L^2(Y)$ of
square integrable functions. This is a unitary representation of $G$
via the left regular action
$$L_g(f)(y):=f(g^{-1}y),\ \ \ \ \ \  g \in G,~ y\in Y,\ f \in
L^2(Y).$$
When $Y$ is compact, the representation $L^2(Y)$ decomposes into a
direct sum of irreducible unitary representations of $G$. The spaces
that will occur naturally in our setting are of the form $Y=G/\Gamma$,
where $\Gamma$ is a discrete co-compact subgroup of $G$.

Recall also that an irreducible unitary representation of $G$ is
called a discrete series if it is a subrepresentation of $L^2(G)$, and
it is called tempered if it occurs in the Hilbert direct integral
decomposition of 
$L^2(G)$. 

\subsection{}Let $\mathcal{H}(G)$ be the Hecke algebra of complex valued, locally constant,
compactly supported functions $f$ on $G$ with the product given by the convolution
\begin{equation}
(f_1*f_2)(g):=\int_G f_1(x)f_2(x^{-1}g)\, dx, \ \ \ \ f_1,f_2 \in
  \mathcal{H}(G).
\end{equation}
Here $dx$ is a Haar measure on $G$.  Let $I$ denote an Iwahori subgroup of $G$. We normalize $dx$ such that the volume of $I$ is
one.

If $(\pi, V)$ is a smooth representation of $G$, we obtain a representation of
$\mathcal{H}(G)$ by defining 
\begin{equation}\label{Hact}
\pi(f)v :=\int_G f(x)~\pi(x)v\, dx \ \ \ \
f \in \mathcal{H}(G), v \in V.
\end{equation}
 Define $\mathcal{H}(G,I)$ to be the subalgebra of $\mathcal{H}(G)$ consisting of functions
bi-invariant under $I$. In the case $G=SU(3)$ (more generally, if $G$
is of simply connected type), the Bruhat-Tits decomposition is $G=\sqcup_{w\in W_{\mathsf{af}}} I w I$. Let $T_1$ and $T_2$ denote the characteristic functions of the double $I$-cosets with representatives $s_1$ and $s_2$, respectively. The algebra
$\mathcal{H}(G,I)$ has the following Iwahori presentation: it is generated by $T_1$ and $T_2$ subject to the
relations
\begin{equation}\label{Iwahori}
T_i^2=(q_i-1)T_i+q_i,\quad i=1,2.
\end{equation}
Denote $V^I:=\{v \in V\, | \, \pi(i)v = v, \ \text{for all} \ i \in I\}$. Then $V^I$ is
stable under the action of $\mathcal{H}(G,I)$ defined in (\ref{Hact}), and we have the natural map (obtained
by restriction)
$
\rho: \Hom_G(V,W) \longrightarrow \Hom_{\mathcal{H}(G,I)}(V^I,W^I),
$
 where
$(\pi,V)$ and $(\pi',W)$ are smooth representations of $G$.

\begin{theorem}[\cite{Bo}]\label{borel} The functor $V\to V^I$ establishes an
  equivalence of categories between the category of smooth admissible
  $G$-representations which are generated by their $I$-fixed vectors
  and the category of finite dimensional $\CH(G,I)$-modules.
\end{theorem}

The inverse functor is described as follows. Let $(\varphi,E)$ be a representation of $\mathcal{H}(G,I)$. Set
\begin{equation}
I(E)=C_c^{\infty}(G/I)\otimes_{\mathcal{H}(G,I)}E,
\end{equation}
 where $C_c^{\infty}(G/I)$ is the space of
compactly supported smooth functions which are right invariant under $I$. The tensor
product makes sense since $\mathcal{H}(G,I)$ acts on $C_c^\infty(G/I)$ by convolution on
the right and acts on the left on $E$ by the representation $\varphi$. We view $I(E)$ as
a $G$-module, where $G$ acts on $C_c^\infty(G/I)$ by left translations. 

\subsection{}
We describe explicitly the irreducible
modules of $\CH=\CH(G,I)$. It will be convenient to use the
Bernstein-Lusztig presentation of $\CH$ (\cite{Lu}) instead of the Iwahori presentation
(\ref{Iwahori}). As a $\mathbb C$-vector space,
$\CH=\CH_W\otimes_{\mathbb C} \CA,$ where $\CH_W=\mathbb C[T]/\langle T^2=(q^{\lambda}-1)
T+q^{\lambda}\rangle,$ $\CA=\bC[\theta],$ and we have the commutation relation:
\begin{equation}\label{eq:1}
\theta T-T\theta^{-1}=(q^{\lambda}-1)\theta+(q^{\frac12(\lambda+\lambda^*)}-q^{\frac12(\lambda-\lambda^*)}).\end{equation} Here $\lambda$ and $\lambda^*$ are, in general, certain parameters, and in the case of $G=SU(3)$, $\lambda=3,$ $\lambda^*=1.$

\medskip

To go from the presentation (\ref{Iwahori}) to (\ref{eq:1}), we will set:
\begin{equation}
T=T_1,\quad \theta=\frac{1}{\sqrt{q_1q_2}}T_1T_2,\quad q^{\lambda}=q_1 \text{ and } q^{\lambda^*}=q_2.
\end{equation}
For the remainder of the section we use the Bernstein-Lusztig presentation. The
subalgebra $\CA$ is abelian, and thus its irreducible representations are
parameterized by the action of $\theta.$ 
Assume that $\theta$ acts by a scalar
$q^\nu,$ $\nu\in \Bbb C.$ Denote the corresponding character of $\CA$
by $\bC_\nu$. Define the principal series modules $X(\nu)$ as
\begin{equation}
X(\nu)=\CH_W\otimes_\CA \bC_\nu.
\end{equation}
Since every irreducible
$\CH$-module has an $\CA$-weight, it can be embedded into a principal
series $X(\nu)$. A
suitable basis for $X(\nu)=\CH_W\otimes_\CA \bC_\nu$ is $\{
(T+1)\otimes 1_\nu, (T-q^{\lambda})\otimes 1_\nu\}.$ In this basis, the
action of the two generators is:
\begin{align}
&T=\left(\begin{matrix} q^{\lambda} &0\\ 0 &{-1}
 \end{matrix}\right);\\\notag
&\theta=\frac 1{q^{\lambda}+1}\left(\begin{matrix}
q^{\lambda-\nu}+q^{\lambda+\nu}+(q^{\frac12(\lambda+\lambda^*)}-q^{\frac 12(\lambda-\lambda^*)})
&q^{\lambda-\nu}-q^{\nu}+(q^{\frac 12(\lambda+\lambda^*)}-q^{\frac12(\lambda-\lambda^*)})\\
q^{-\nu}-q^{\lambda+\nu}-(q^{\frac 12(\lambda+\lambda^*)}-q^{\frac 12(\lambda-\lambda^*)})
& q^{-\nu}+q^{\nu}-(q^{\frac 12(\lambda+\lambda^*)}-q^{\frac12(\lambda-\lambda^*)})
\end{matrix}\right).
\end{align}
To find this matrix for $\theta$, one multiplies the basis elements
on the left by $\theta$ and then commutes past $T$ using equation (\ref{eq:1}).
The eigenvalues of $\theta$ are $q^\nu$ and $q^{-\nu}.$ These are the
$\CA$-weights of the module $X(\nu).$

 We summarize the classification
of irreducible $\CH$-modules.

\begin{proposition}\label{irrH}

\begin{enumerate}

\item Every irreducible $\CH$-module appears as a subquotient of a two-dimensional
  principal series $X(\nu),$ where $\nu\in \bC/(2\pi i/\log q)$.\medskip

\item The module $X(\nu)$ is irreducible unless $\nu\in\{\pm
  \frac{\lambda+\lambda^*}2,\pm \frac{\lambda-\lambda^*}2+\frac {\pi
    i}{\log q}\}.$ \medskip

\item The one dimensional $\CH$-modules are:
\begin{equation}\label{eq:2}
\begin{aligned}
&\mathsf{St}=(T=-1,\ \theta=q^{-\frac 12(\lambda+\lambda^*)});
&\mathsf{ds}=(T=-1,\ \theta=-q^{\frac 12(\lambda^*-\lambda)});\\
&\mathsf{sph}=(T=q^{\lambda}, \ \theta=q^{\frac 12(\lambda+\lambda^*)}); 
&\mathsf{nt}=(T=q^{\lambda},\ \theta=-q^{\frac 12(\lambda-\lambda^*)}).
\end{aligned}
\end{equation}

\end{enumerate}

\end{proposition}

\begin{proof} To find the one-dimensional modules, and implicitly
  determine when $X(\nu)$ is reducible, notice that if $T$ acts by a scalar, there are two possibilities, either $T=-1$ or
$T=q^{\lambda}$. Solving in the commutation relation (\ref{eq:1}) for
 $\theta$, we obtain these four one-dimensional representations.
\end{proof}

The center of $\CH$ can be shown to equal $Z(\CH)=\mathbb
C[\theta+\theta^{-1}]$. (This is an easy particular case of a more
general result of Bernstein and Lusztig.) 
Therefore, if $(\pi,V)$ is a subquotient of
$X(\nu)$, the central character of $\pi$ is determined by $\nu$.

\subsection{}
Later we will need to know the unitary irreducible modules of
$\CH$. The algebra $\CH=\CH(G,I)$ has a natural $*$-operation defined
on functions by
\begin{equation}
f^*(g):=\overline{f(g^{-1})},\quad f\in \CH(G,I).
\end{equation}
Calculated on the generators $T,\theta$, this becomes:
\begin{equation}
T^*=T^{-1},\qquad \theta^*=T\theta^{-1}T^{-1}. 
\end{equation}

\begin{definition}  \begin{enumerate}
\item We say that an $\CH$-module $(\pi,V)$ is ($*$-)unitary  if it admits a
Hermitian form $\langle\ ,\ \rangle$ such that
\begin{equation}
\langle\pi(x)v,w\rangle=\langle v,\pi(x^*)w\rangle,\quad \text{ for
  all }x\in\CH, v,w\in V.
\end{equation}
\item (Casselman's criterion) An $\CH$-module $(\pi,V)$ is called
  tempered (resp. discrete series) if $|\chi|\le 1$
  (resp. $|\chi|<1$), for every $\CA$-weight $\chi$ of $\pi.$
\end{enumerate}
\end{definition}


\begin{proposition}\label{unitary}
\begin{enumerate}
\item In the correspondence $V\to V^I$ of Theorem \ref{borel}, the
  $G$-representation $V$ is tempered (resp. discrete series, unitary) if and only if the
  $\CH(G,I)$-representation $V^I$ is tempered (resp. discrete series, unitary).
\item The discrete series $\CH$-modules are the one
  dimensional modules $\mathsf{St}$  and $\mathsf{ds}.$ The tempered
$\CH$-modules,  that are not discrete series, are the direct summands of $X(\nu)$,
  for $\Re\nu=0$. 
\item The unitary $\CH$-modules are the subquotients of the principal
  series $X(\nu)$ in one of the following cases:
\begin{itemize}
\item $\nu\in\bR$ and $|\nu|\le \frac{\lambda+\lambda^*}2$;
\item $\Re\nu=0;$
\item $\Im\nu=\frac{\pi i}{\log q}$ and $|\Re\nu|\le \frac{\lambda-\lambda^*}2.$
\end{itemize}
\end{enumerate}
\end{proposition}

\begin{proof} (1) The claim about tempered modules and discrete series
  is well-known (see \cite{Cas}). The correspondence for unitary
  modules is a particular case of Theorem 1 in \cite{BC}.

(2) This is obvious from Proposition \ref{irrH}, just recall that with our notation
$\lambda>\lambda^*.$

(3) One way to prove this is to make use of \cite{BC}, where we can reduce this question to one about Lusztig's
graded Hecke algebra \cite{Lu}. (It is of course possible to prove the
claim directly, without appealing to the graded Hecke algebra.)

Let $\mathbb H_\mu$ be the $\mathbb
C$-algebra with unit generated by $s$ and $\epsilon$ subject to the
relations
\begin{align}
s^2=1,\quad \epsilon\cdot s+s\cdot \epsilon=2\mu,
\end{align}
where $\mu$ is a nonnegative scalar. This algebra has a $*$-operation as
well, defined on the generators by 
\begin{equation}s^*=s,\quad \epsilon^*=\epsilon+\mu s,\end{equation}
and again we can talk about unitary modules. The classification of the
irreducible $\mathbb H_\mu$-modules is similar, but simpler, than the
one for $\CH$-modules. More precisely, denote by $\mathbb
C[W]=\mathbb C[s]$ and
$\mathbb A=\mathbb C[\epsilon]$, and define the principal series $\overline
X(\nu)=\mathbb C[W]\otimes_{\mathbb C} \mathbb C_\nu$, where
$\mathbb C_\nu$ is a character of
$\mathbb A$, $\nu\in\mathbb C.$ Then, we have:
\begin{itemize}
\item every irreducible $\mathbb H_\mu$-module is a subquotient of a
  $\overline X(\nu)$;
\item $\overline X(\nu)$ is irreducible unless $\nu=\pm\mu$;
\item the one-dimensional irreducible $\mathbb H_\mu$-modules are
  $(s=1,\epsilon=\mu)$ and $(s=-1,\epsilon=-\mu)$. 
\end{itemize}
We say that an $\mathbb
H_\mu$-module has real central character if the eigenvalues of
$\epsilon$ are real scalars, or equivalently, if the parameter $\nu$
is in $\mathbb R.$ 

The unitary modules are easily determined too. A hermitian form on
$\overline X(\nu)$ is equivalent with an intertwining operator between
$X(\nu)$ and the hermitian dual of $X(\nu)$, which is
$X(-\overline\nu).$ When $\nu\ge 0$, the operator is 
\begin{equation}
A(\nu):
\overline X(\nu)\to \overline X(-\nu),\quad A(\nu)(x\otimes
1_\nu)=\frac 1{\nu+\mu} (\epsilon\cdot s+\mu)\cdot x\otimes 1_{-\nu}. 
\end{equation}
Using the basis $\{(1+s)\otimes 1_\nu,(1-s)\otimes 1_\nu\}$, one can
immediately find that
$[A(\nu)]=\left(\begin{matrix}1&0\\0&\frac{\mu-\nu}{\mu-\nu}\end{matrix}\right)$. This
    means that $A(\nu)$ is positive semidefinite if and
only if $-\mu\le\nu\le\mu.$

The relation to $\CH$ is as follows. Fix $\xi\in i(\mathbb R/\log
q)$. Then \cite{Lu} defines an ideal $\mathcal I_\xi$ of $\CA$ such
that the associated graded object to the filtration $\CH\supset
\CH\cdot \mathcal I_\xi\supset \CH\cdot \mathcal I^2_\xi\supset\dots$
is a graded Hecke algebra of the type defined here. The point is that one
obtains a bijection between irreducible $\CH$-modules parameterized by
$\nu$ such that $\Im\nu=\xi$ and $\mathbb H$-modules parameterized by
$\Re\nu$. This correspondence preserves unitarity too. The two interesting
cases are $\xi=0$, in which case the corresponding graded algebra is
$\mathbb H_\mu$, $\mu=\frac{\lambda+\lambda^*}2$, and the case
$\xi=\frac {\pi i}{\log q}$ when the corresponding graded algebra is
$\mathbb H_\mu$, $\mu=\frac{\lambda-\lambda^*}2$. This explains the
unitary ``complementary series'' in our proposition. (When $\xi\neq
0,\frac {\pi i}{\log q}$, the corresponding graded Hecke algebra is
not one of the $\mathbb H_\mu$'s, but rather just an abelian
algebra. Consequently, its unitary dual with ``real central
character'' is a point, corresponding to the tempered module $X(\nu)$, $\Re\nu=0.$)
\end{proof}

\section{Adjacency operators and quotients of the $SU(3)$-tree}

We retain the notation from the previous sections. 
The Bruhat-Tits tree of $SU(3)$ is a $(q^3+1,q+1)$-biregular graph, where $q$ is the cardinality of   the residue field of $K$. We denote it by $\tilde X.$ (One obtains the same graph if one considers $U(3)$ instead of $SU(3).$) 

Recall that $I$ is an Iwahori subgroup. Set $U_i:=I \cup Is_iI$, $i=1,2$. The subgroups $U_i$, $i=1,2$, contain $I$ properly and are
representatives of the $G$-conjugacy classes of maximal (open) compact
subgroups. In fact, $I=U_1\cap U_2$.  Define a length function on $G$ as in
\cite{H-H} and let $U(=U_1)$ be the set of elements of length $0$. Let $\Gamma$ be a
discrete co-compact subgroup of $G$ as in \cite{H-H}, such that $\Gamma$ is torsion
free, $\Gamma \cap x^{-1}Ux=\{1\}$ for any $x \in G$, and $\#(U\backslash G / \Gamma)
< \infty$.

Let $X$ be the quotient graph $\tilde{X}/\Gamma$. Then $X$ is a finite, connected,
biregular, bipartite graph with valencies $q^3+1$ and $q+1$,  and therefore the discussion in the previous
sections applies to $X$. We would like to investigate when $X$ is a
Ramanujan bigraph. For this, we will make use of the Iwahori-Hecke
algebra $\CH=\CH(G,I)$ defined before.

\subsection{}\label{sec:4.1} We summarize the results from  \cite[\S 5,6]{H}
establishing the relationship  between the algebra generated by the
colored edge adjacency operators and $\CH(G,I).$ 

Let $\mathbb{Z}[E(\tilde{X})]$ be the free $\mathbb{Z}$-module
over the set $E(\tilde{X})$ of edges of $\tilde{X}$. We define $T_1,T_2$ to be the
elements of $\End(\mathbb{Z}[E(\tilde{X})])$ given by 
\begin{equation}
T_i(e):= \displaystyle{\sum_{e'\in
\tilde{E}_i(e)\setminus \{e\}}} e'\ \ \ (i=1,2),
\end{equation}
 where $\tilde{E}_i(e)$ is the set of
edges in $\tilde X$ incident to the vertex of $e$ that lies in $V_i$.

Let $\bC[E(X)]$ be the space of $\mathbb{C}$-valued functions on $E(X)$, the edges of
$X$. It carries an inner product 
\begin{equation}\label{innprod}
(f,f'):=\sum_{e\in E(X)} (f(e),f'(e))_\bC,
\end{equation}
where $(~,~)_\bC$ is the usual inner product in $\bC$.
Since the action of $\Gamma$ preserves the incidences in $\tilde{X}$, the
operators $T_1, T_2$ induce naturally endomorphisms on $\bC[E(X)]$ (denoted $T_1,T_2$ again),
\begin{equation}(T_if)(e):=\sum_{e' \in
    E_i(e)}f(e')-f(e)\ \ \ (i=1,2).
\end{equation}
Moreover, $T_1,T_2$ are isometries with respect to the inner product (\ref{innprod}). 
Recall the Zeta function $Z_X$ from definition \ref{zeta}. By the Main Theorem (I) of \cite{H}, we have 
\begin{equation}\label{zeta2}
Z_X(u)^{-1}=\det(I-{T_1T_2}u). 
\end{equation}
Let $\mathbb{C}[T_1,T_2]$ be the $\mathbb C$-subalgebra of
$\End_\bC(\mathbb Z[E(X)])$ generated by $T_1, T_2$.
It is a non-commutative ring of polynomials in $T_1, T_2$ with fundamental relations
$$T^2_i=(q_i-1)T_i+q_i, \ \ \ (i=1,2).$$
This means that $\mathbb C[T_1,T_2]\cong \CH(G,I).$
In the previous section (Proposition \ref{irrH}), we saw that the
irreducible $\CH(G,I)$-modules lie in two-dimensional principal series
$X(\nu),$ where $\nu$ is in the parameter space $\bC/(2\pi i/\log
q)$. To every irreducible Hecke module $\varphi$, one associates a characteristic
polynomial:
\begin{equation}
\begin{aligned}
p_\varphi(u)&=\det(1-\varphi(T_1T_2)u)=\det(1-\sqrt{q_1q_2}\varphi(\theta)
u)\\&=1-\sqrt{q_1q_2}~\Tr_\varphi(\theta)+q_1q_2 u^2,
\end{aligned}
\end{equation}
when $\varphi$ is two dimensional; here $\theta$ is the generator from (\ref{eq:1}). Notice that the one dimensional modules have the following characteristic polynomials:
\begin{equation}
\begin{aligned}
&p_{\mathsf{St}}(u)=1-u; &p_{\mathsf{ds}}(u)=1+q_2u;\\
&p_{\mathsf{sph}}(u)=1-q_1q_2u; &p_{\mathsf{nt}}(u)=1+q_1u.
\end{aligned}
\end{equation}
We need to determine  which representations $\varphi$ of $\CH$ have the property
that $p_\varphi(u)$ occurs in $Z_X(u)^{-1}$. Notice that from
(\ref{prod}), one immediately sees that $p_{\mathsf{St}}(u)$ occurs
with multiplicity exactly $r$, $p_{\mathsf{sph}}(u)$ has multiplicity one, while $p_{\mathsf{ds}}(u)$ occurs with
multiplicity at least $n_2-n_1.$ More precisely,
$p_{\mathsf{ds}}(u)$ occurs in $Z_X(u)^{-1}$ with multiplicity exactly $n_2-n_1$ if and only if $X$ is weakly Ramanujan. The only way the multiplicity of $p_{\mathsf{ds}}(u)$ is greater then $n_2-n_1$ is if $p_{\mathsf{nt}}(u)$ occurs in $Z_X(u)^{-1}$.

\subsection{}
Recall the space
$L^2(G/\Gamma)$ with the left regular representation of $G$. Consider
the subspace of $I$-invariant vectors $L^2(G/\Gamma)^I=L^2(I\backslash
G/\Gamma)=\bC[E(X)].$ If $\pi$ is a unitary irreducible representation
of $G$, let $m_{\Gamma}(\pi)$ represent the multiplicity of $\pi$ in
$L^2(G/\Gamma)^I$. The two discrete series of $G$, whose $I$-fixed
vectors form the $\mathcal H(G,I)$-modules $\mathsf{St}$ and
$\mathsf{ds}$, occur in this space, for every such
$\Gamma.$ Since $G/\Gamma$ is compact, the trivial representation also occurs.

Let  $(\pi,V)$ be an irreducible unitary
smooth representation of $G$ such that $V^I\neq \{0\}$. 
Let $(\varphi,
V^I)$ be the corresponding irreducible representation of
$\mathcal{H}(G,I)$. By Proposition \ref{unitary}, this is a unitary
finite dimensional module of $\CH(G,I).$

Putting these together, one finds (Main Theorem (IV) of \cite{H}):
\begin{equation}\label{mult}
m_{\Gamma}(\pi)= \, \mbox{multiplicity of}\ p_{\varphi}(u) \ \mbox{in}\
Z_X(u)^{-1}.
\end{equation}
If $m_\Gamma(\pi)>0$, and $\varphi\notin\{\mathsf{St},\mathsf{ds}\}$, we
see that the corresponding factor in $Z_X(u)^{-1}$ is
$p_\varphi(u)=1-\sqrt{q_1q_2}~\Tr_\varphi(\theta)+q_1q_2u^2$ on one
hand, but on the other is of the form
$(1-(\lambda_j^2-q_1-q_2)u+q_1q_2u)$. This means that
$\Tr_\varphi(\theta)=\frac {\lambda_j^2-q_1-q_2}{\sqrt{q_1q_2}}$, for
some eigenvalue $\lambda_j$ of $\Ad(X)$. The Ramanujan condition
(\ref{ramanujan}) is therefore equivalent to: 
\begin{equation}
|\Tr_\varphi(\theta)| \leq 2, \text{ \emph{i.e.}, $\varphi$ is tempered (by Proposition \ref{unitary}).}
\end{equation}
We phrase this condition as follows.

\begin{conjecture}[Ramanujan conjecture]\label{conj} Every nontrivial irreducible
  unitary $\mathcal H(G,I)$-module
that appears in the decomposition of $L^2(G/\Gamma)^I=L^2(I\backslash
G/\Gamma)$ is tempered.
\end{conjecture}

Recall that the explicit description of unitary $\mathcal H(G,I)$-modules is given in Proposition \ref{unitary}. We have obtained the following criterion:

\begin{theorem}\label{t:crit} Let $G=SU(3)$ be the unitary
  group in three variables defined in section \ref{sec:su3}. Let $\Gamma$ be a discrete, co-compact subgroup of $G$ which acts on $G$ without fixed points. If $\tilde{X}$ is the Bruhat-Tits tree associated with $G$, then the building quotient $X=\tilde{X}/\Gamma$ is a Ramanujan bigraph if and only if $G$ satisfies Conjecture \ref{conj}. 

\end{theorem}

\subsection{}\label{sec:3.3}
 In this section, let $k$
denote a global field of characteristic $0$, and let $E/k$ be a
separable quadratic extension. Denote by $\ \bar{ }\ $ the conjugation of $E$
with respect to $k$. Let $D$ be a simple $9$-dimensional algebra with
center $E$ and norm $N$, and let $\alpha:D\to D$ be an
anti-automorphism such that $\alpha(x)=\overline x$, for all $x\in E$ (\emph{i.e.}, an
automorphism of second kind). For the classification of pairs
$(D,\alpha)$ we refer to \cite{KMRT}, especially Theorems (3.1) and
(19.6). Every such $D$ is a cyclic algebra over $E$
(Theorem (19.2) and Proposition (19.15) in \cite{KMRT}). More
precisely, every $(D,\alpha)$ arises as follows. Let $L$ be a $3$-dimensional algebra over $E$, which is $\mathbb
Z/3\mathbb Z$-Galois over $E$ and $S_3$-Galois over $k$. This means that $L$ has an automorphism $\sigma$
of order $3$ and an automorphism $\iota$ of order $2$, such that
$\langle\sigma,\iota\rangle\cong S_3$, $\iota|_E=\bar\ $, and the
fixed points of $\sigma$ in $L$ are $L^\sigma=E$, while the fixed
points of $\langle\sigma,\iota\rangle$ in $L$ are
$L^{\langle\sigma,\iota\rangle}=k$. Notice that
$\sigma\circ \iota=\iota\circ\sigma^2$ on $L.$ Then
set 
\begin{equation}\label{eq:D}
\begin{aligned}
&D=L\oplus Lz\oplus Lz^2,\text{ with}\\
&z\ell=\sigma(\ell)z,\quad z^3=a,
\end{aligned}
\end{equation}
for some fixed $a\in k^\times$, and set the involution $\alpha$ to be
\begin{equation}\label{eq:L}
\alpha(z)=z,\quad \alpha(\ell)=\iota (\ell),\text{ for all }\ell\in L.
\end{equation}

 Define $\mathbb G$ to be an inner form
of $SU(3)$ determined by $(D,\alpha)$, more precisely $\mathbb
G=\{g\in D^*: \alpha(g)g=1, N(g)=1\}.$ For example, if $D\cong M_3(E),$
meaning in (\ref{eq:D}) that $L\cong E^3$,
then $\alpha(g)=\Phi ~\overline g^T ~\Phi^{-1}$, for some Hermitian form
$\Phi$. The quasisplit form from section \ref{sec:su3} is such an
example. For the explicit connection with the realizations
(\ref{eq:D}, \ref{eq:L}), we refer to Example (19.17) in \cite{KMRT}. 

\smallskip

For every place $v$, we denote $\mathbb G_v=\mathbb
G(k_v)$ and $D_v=D\otimes_k k_v$. For the finite places $v$ of $k$, we have
the following possible cases (see \cite{Sch} or \S1.9
in \cite{R})
for the group $\mathbb G_v=\{g\in D_v^*: \alpha(g)g=1, N(g)=1\}$ over
the local field $k_v$:

\begin{enumerate}
\item[(a)] if $v$ splits into $v=w\overline w$ in $E$, then
  $D_v=D_w\oplus D_{\overline w}$, and $\mathbb G_v\cong
  D_w^1\cong D_{\overline w}^1$ (norm one units);
\item[(b)] if $v$ remains prime in $E$, then $\mathbb G_v\cong G$, where $G$ is the quasisplit unitary group over $k_v$ defined in section \ref{sec:su3}.
\end{enumerate}

 In order to use the results of \cite{R},  we restate the
previous criterion, Theorem \ref{t:crit}, from a global perspective.

\begin{theorem}\label{main} Let $v$ be a finite place of $k$ with
  $E_v/k_v$ unramified, and such that at the place $v$, $\mathbb
  G_v\cong G$.  Let $\Gamma_v$ be a discrete, co-compact subgroup of
  $\mathbb G_v$ which acts on $\mathbb G_v$ without fixed points. If
  $\tilde{X_v}$ is the Bruhat-Tits tree associated with $\mathbb G_v$,
  then the building quotient $X_v=\tilde{X_v}/\Gamma_v$ is a Ramanujan
  bigraph if and only if $\mathbb G_v$ satisfies Conjecture \ref{conj}. 

\end{theorem}

By the Strong Approximation Theorem (\cite{K}), verifying Conjecture \ref{conj} is equivalent with deciding which representations of $\mathbb G_v=G$ with $I$-fixed vectors occur as local components of the automorphic representations of $\mathbb G$ over the adeles. We explain this connection now.

 In order to simplify notation, let us assume that $k=\mathbb Q$, and
 we will denote a prime, as usual, by $p$. Assume that $E/\mathbb Q$ is imaginary.
  Let $\mathbb A$ denote the ring of adeles of $\mathbb Q$, and let
  $\mathbb{A}_p$ be the ring of adeles  without the
  factor at the place $p$. Let $\mathbb Z$ be the ring of integers of
  $\mathbb Q$. 

Fix a finite prime $p$ as in Theorem \ref{main}. Since we are looking at forms of unitary groups of {\it
    odd} order, there exists a $\mathbb Q$-group $\mathbb G$ such that
  $\mathbb G_p\cong G$; in particular, this means that $\mathbb G_p$
  is not compact. (In fact, $\mathbb G_v\cong G$ at almost all finite places.) Moreover, we may require that $\mathbb G(\mathbb R)$
  be compact. Such groups $\mathbb G$ exist, see for example \cite[\S3.3]{CHT}. 

 Let $\mathbb Z[p^{-1}]$
denote the subring of $\mathbb Q$ consisting of all rational numbers
whose denominators are powers of $p$. One thinks of $\mathbb
Z[p^{-1}]$ as being embedded diagonally into $\mathbb R\times \mathbb
Q_p$. A theorem of Borel \cite{Bo2} implies that $\Gamma_p=\mathbb G(\mathbb Z[p^{-1}])$ is a lattice in $\mathbb G(\mathbb R)\otimes  \mathbb G_p$. For every positive integer $n$ coprime to $p$, define 
\begin{equation}\label{gamma}
\Gamma(n)=\ker (\mathbb G(\mathbb Z[p^{-1}])\to \mathbb G(\mathbb Z[p^{-1}]/n\mathbb Z[p^{-1}]), \quad \Gamma_p(n)=\Gamma(n)\cap \mathbb G_p.
\end{equation}
Therefore $\Gamma_p(n)$ is a discrete cocompact lattice in $\mathbb G_p$.

Assume that the factorization of $n$ in $\mathbb Z$ into primes is $n=\prod_{i=1}^r p_i^{d_i}.$ Define:
\begin{equation}
\begin{aligned}
&K_{p_i}=\ker(\mathbb G(\mathbb Z_{p_i})\to \mathbb G(\mathbb Z_{p_i}/{p_i}^{d_i}\mathbb Z_{p_i})),\ \  i=1,\ldots, r;\\
&K_\ell=\mathbb G(\mathbb Z_\ell),\ \ \ell\neq p_i, p, \  \ell<\infty;\\
&K_p=I_p\ \ (\text{the Iwahori subgroup}); \\
&K_\infty=SU(3) \ \ (\text{the compact unitary group over }\mathbb R).
\end{aligned}
\end{equation}
Set 
\begin{equation}
K^n=\prod_{l\le \infty} K_\ell
\quad \text{ and }\quad
K^n_{(p)}=\prod_{l\le \infty, \ell\neq p} K_\ell.
\end{equation}
These are  compact open subgroups of $\mathbb G(\mathbb{A})$  and  $\mathbb G(\mathbb{A}_p)$  respectively. If we
assume that $\mathbb G(\mathbb R)=K_\infty$, in other words that
$\mathbb G(\mathbb R)$ is compact, then $\mathbb G(\mathbb Q)\cap K^n_{(p)}=\Gamma_p(n)$. 

By \cite[Theorem 5.1]{Bo2}, The number of double cosets in $$K^n\setminus \mathbb G(\mathbb A)/\mathbb G(\mathbb Q)$$is finite and, therefore, the number of double cosets in $$K^n_{(p)}\mathbb G_p\backslash \mathbb G(\mathbb A)/\mathbb G(\mathbb Q)$$ is finite. Let $\{x_1, \ldots, x_s\} $ be a set of representatives of these cosets. We have

\begin{equation}\mathbb G(\mathbb A)= \bigcup_{i=1}^s  \left( K^n_{(p)}\mathbb G_p \right)
x_i \mathbb G(\mathbb Q).\end{equation}
Consider the group $\Gamma' _{i,n}=K^n_{(p)}\mathbb G_p \cap x_i
\mathbb G x_i^{-1}$. Note that when $x_i$ is the identity, $\Gamma'
_{i,n}$ is precisely $\Gamma(n)$ defined in (\ref{gamma}). Since
$K_{\infty}$ is compact, the projection of $\Gamma' _{i,n}$ on
$\mathbb G_p$, denoted by $ \Gamma' _{i,p,n}$, remains a discrete
subgroup. Each group $\Gamma' _{i,p,n}$ is finitely generated. Then,
by \cite[Lemma 8]{Sel}, $\Gamma' _{i,p,n}$ has a normal subgroup
$\Gamma _{i,p,n}$ of finite index which has no nontrivial elements of
finite order. Thus, each element of $\Gamma _{i,p,n}$ different from
the identity  acts on $I_p\backslash \mathbb G_p$ without fixed
points. As noted before, when $x_i$ is the identity element,
$\Gamma_{i,p,n}$ is precisely  $\Gamma_p(n)$. We have

\begin{equation}
L^2(K_p^n\backslash\mathbb G(\mathbb A)/\mathbb G(\mathbb Q))=\bigotimes_{i=1}^s L^2(\mathbb G_p/\Gamma_{i,p,n}).
\end{equation}
In other words, every irreducible $\mathbb G_p$-representation that occurs in \newline $L^2(\mathbb G_p/\Gamma_{i,p,n})$ must be the local factor of an automorphic representation of $\mathbb G$.
Define
\begin{equation}\label{tree}
X_{i,p,n}:= K_p\backslash\mathbb G_p/ \Gamma_{i,p,n}.
\end{equation}
 In \cite{R}, Rogawski completely classifies the automorphic
 representations of $\mathbb G$. (A convenient reference for the
 needed results from \cite{R} is  \cite[\S 11]{Ba}.) Recall that we
 are assuming that at the place $p$, $\mathbb G_p=G$. In turns out
 that the only nontrivial non-tempered automorphic representation of
 $\mathbb G_p$ is the unitary representation whose $I$-fixed vectors
 form the $\mathcal H(G,I)$-module $\mathsf{nt}$ (notation as in
 Proposition \ref{irrH}). By the discussion at the end of
 \S\ref{sec:4.1}, this result of \cite{R} implies the following equivalence:

\begin{corollary}
Assume that $\mathbb G(\mathbb R)$ is compact and that $\mathbb G_p=G$. Recall the finite bigraphs $X=X_{i,p,n}$ constructed in (\ref{tree}). Then such a graph $X$ is Ramanujan if and only if $X$ is weakly Ramanujan.
\end{corollary}
Recall that ``weakly Ramanujan'' is a simple condition on the multiplicity of $0$ as an eigenvalue of $\Ad(X)$ (Definition \ref{d:ram}).

\smallskip

Moreover, again from \cite{R}, we have:

\begin{theorem}[\cite{R}, Theorem 14.6.3] Assume that $D\neq M_{3}(E)$. Then no representation of the form $\mathsf{nt}$ occurs as the local component of an automorphic representation of $\mathbb G$. 
\end{theorem}

Therefore comparing this with Theorem \ref{main}, we have the following corollary. 

\begin{corollary}\label{c:main} Assume that $D\neq M_3(E)$ and that
  $\mathbb G(\mathbb R)$ is compact.  Then the family of finite tree quotients $X_{i,p,n}$ from (\ref{tree}) is an infinite family of Ramanujan bigraphs. 
\end{corollary}

\end{document}